\documentclass[11pt]{article}

\def\TEXPAD{0}

\usepackage{amsmath,amssymb,amsthm}
\usepackage[mathscr]{eucal} 
\usepackage{cite} 
\usepackage{hyperref}
\usepackage[capitalize, nameinlink]{cleveref}
\usepackage{enumitem}
\usepackage[usenames,dvipsnames]{xcolor} 
\usepackage{tikz}
\usepackage[font=small,labelfont=bf, margin=.5in]{caption}

\colorlet{prettygreen}{ForestGreen!60!LimeGreen}

\usetikzlibrary{calc, arrows.meta}
\tikzset{vtx/.style={circle, fill, inner sep=1.5pt}}
\tikzset{openvtx/.style={circle, draw, inner sep=1.5pt}}

\newtheorem{theorem}{Theorem}[section]
\newtheorem{lemma}[theorem]{Lemma}
\newtheorem{proposition}[theorem]{Proposition}

\newtheorem{claim}[theorem]{Claim}
\newtheorem*{claim*}{Claim}

\theoremstyle{definition}
\newtheorem{definition}[theorem]{Definition}

\newtheorem{question}[theorem]{Question}
\newtheorem{conjecture}[theorem]{Conjecture}

\theoremstyle{remark}

\newenvironment{prop}{\begin{proposition}}{\end{proposition}}
\newenvironment{dfn}{\begin{definition}}{\end{definition}}

\if\TEXPAD0
	\crefname{claim}{Claim}{Claims}
\fi

\newcommand{\mydots}[1]{\,\overset{_{#1}}{\cdots}\,}

\newcommand{\al}{\alpha}
\newcommand{\eps}{\epsilon}
\newcommand{\E}{{\mathbb E}}
\newcommand{\A}{{\mathcal A}}
\newcommand{\DD}{{\mathcal D}}
\newcommand{\FF}{{\mathcal F}}
\newcommand{\HH}{{\mathcal H}}
\renewcommand{\SS}{{\mathcal S}}
\newcommand{\RP}{{\mathbb R\mathbb P}^2}
\newcommand{\Vo}{V^\circ}

\newcommand{\disk}{\mathbb D^2}
\newcommand{\diskm}{\mathbb D^{2-}}
\newcommand{\sph}{\mathbb S^2}
\newcommand{\sone}{\mathbb S^1}
\newcommand{\svs}{\sone\vee\sone}
\newcommand{\tor}{\mathbb T^2}

\DeclareMathOperator{\ex}{ex}
\newcommand{\X}{X}
\newcommand{\exh}{\ex_{\hom}}

\newcommand{\kptz}{Kupavskii, Polyanskii, Tomon, and Zakharov}

\title{The Tur\'an Number of Surfaces}
\author{Maya Sankar}
\date{October 20, 2022}

\begin{document}
\maketitle

\begin{abstract}
We show that there is a constant $c$ such that any 3-uniform hypergraph $\HH$ with $n$ vertices and at least $cn^{5/2}$ edges contains a triangulation of the real projective plane as a subgraph. This resolves a conjecture of \kptz. Furthermore, our work, combined with prior results, asymptotically determines the Tur\'an number of all surfaces.
\end{abstract}

\section{Introduction}

Tur\'an-type questions are fundamental in the study of extremal combinatorics. Given a fixed $r$-uniform hypergraph $\FF$, its \emph{Tur\'an number} $\ex(n,\FF)$ is the maximum number of edges in an $r$-uniform hypergraph $\HH$ on $n$ vertices which does not contain $\FF$ as a subgraph. Estimating Tur\'an numbers for hypergraphs remains a largely open problem; we refer the reader to the surveys \cite{Ke11,MuPiSu11,FuSi13} for a general overview.

This paper investigates a topological variant of this extremal problem. Any $r$-uniform hypergraph $\HH$ may be viewed as an $(r-1)$-dimensional simplicial complex whose facets are the edges of $\HH$. Similarly, one may ask if any subgraph of $\HH$ is homeomorphic to a given $(r-1)$-dimensional simplicial complex $X$. This geometric perspective yields many natural generalizations of graph properties to higher dimensions. For example, one analogue of Hamiltonian cycles in 3-uniform hypergraphs that has received some attention (see \cite{GHMN22,LuTe19}) is a spanning sub-hypergraph homeomorphic to the 2-sphere. Additionally, one is naturally interested in the following extremal quantity. Let $X$ be a closed $(r-1)$-dimensional manifold. Denote by $\exh(n,X)$ the maximum number of edges in a $r$-uniform hypergraph $\HH$ on $n$ vertices such that no subgraph of $\HH$ is homeomorphic (as a simplicial complex) to $X$. This is the \emph{Tur\'an number} of the topological space $X$.

As part of his program in high-dimensional combinatorics, Linial \cite{Li18} asked for the asymptotics of $\exh(n,X)$ when $r\geq 3$. Linial's question was partially motivated by the work of S\'os, Erd\H os, and Brown \cite{SoErBr73} some decades prior, which showed that $\exh(n,X)=\Theta(n^{5/2})$ when $X$ is the 2-sphere $\sph$. Linial \cite{Li18} sketched a new proof of the lower bound $\exh(n,\sph)=\Omega(n^{5/2})$ which generalized to all closed, connected 2-manifolds $X$; this proof is given rigorously in \cite[\S 2]{KPTZ21}. We call such a 2-manifold a \emph{surface}. 

All surfaces fall into one of three categories: the sphere $\sph$, the connected sum of $g\geq 1$ tori, or the connected sum of $k\geq 1$ real projective planes. Until recently, it was unknown if the lower bound of $n^{5/2}$ was asymptotically tight for the latter two classes. Indeed, Linial \cite{Li08,Li18} repeatedly conjectured a matching upper bound for the torus $\tor$, i.e.\ that $\exh(n,\tor)=O(n^{5/2})$. \kptz \cite{KPTZ21} proved Linial's conjecture in 2020. Additionally, they showed that if two surfaces $X_1, X_2$ satisfy $\exh(n,X_i)=O(n^{5/2})$, their connected sum $X_1\# X_2$ also satisfies $\exh(n,X_1\#X_2)=O(n^{5/2})$, thereby extending the upper bound $\exh(n,X)=O(n^{5/2})$ to orientable surfaces of the form $X=\tor\#\cdots\#\tor$. They were unable to derive the corresponding result for any non-orientable surfaces, but conjectured that the same bound applies to all surfaces.

Our main result is the resolution of this conjecture. We show that Linial's lower bound is asymptotically tight for the real projective plane $\RP$.
\begin{theorem}\label{mainthm}
We have $\exh(n,\RP)=O(n^{5/2})$.
\end{theorem}
\noindent By \kptz's result on connected sums, this bound generalizes to all non-orientable surfaces $X=\RP\#\cdots\#\RP$. Combining our work with the results of S\'os, Erd\H os, and Brown \cite{SoErBr73} for the sphere and \kptz \cite{KPTZ21} for all other orientable surfaces, we completely determine the asymptotics of $\exh(n,X)$ for any surface $X$.
\begin{theorem}\label{genthm}
Let $X$ be any surface. Then $\exh(n,X)=\Theta(n^{5/2})$, where the constant coefficients may depend on the surface $X$.	
\end{theorem}

The remainder of this paper is organized as follows. \Cref{s:RP} describes how simpler substructures of a 3-uniform hypergraph $\HH$ may be combined to obtain a subgraph homeomorphic to $\RP$. \Cref{s:adm} introduces probabilistic methods, proving three lemmas used to bound the probability that a random subset of $V(\HH)$ contains one of these structures. We prove \cref{mainthm} in \cref{s:main} and discuss some open problems in \cref{s:conclusion}.

\paragraph*{Notation.} Let $\HH$ be a 3-uniform hypergraph. 
If $u$ is a vertex of $\HH$, its \emph{link graph} $\HH_u$ is the graph on $V(\HH)\setminus\{u\}$ whose edges $vw$ correspond to 3-edges $uvw\in E(\HH)$. For distinct vertices $u$ and $u'$ of $\HH$, we write $\HH_{u,u'}=\HH_u\cap\HH_{u'}$; that is, $\HH_{u,u'}$ is the graph on $V(\HH)\setminus\{u,u'\}$ with edge set $E(\HH_u)\cap E(\HH_{u'})$.

For convenience, we do not always distinguish between a hypergraph and its associated simplicial complex. 
For instance, we call $\HH$ a disk, or say $\HH$ is homeomorphic to the disk, if its associated simplicial complex is homeomorphic to the disk $\disk$. 
Where the distinction is essential, we write $\X(\HH)$ for the simplicial complex associated to the hypergraph $\HH$.

\section{Finding $\RP$ in a Hypergraph}\label{s:RP}

Our proof of \cref{mainthm} begins by identifying conditions under which a 3-uniform hypergraph $\HH$ contains a subgraph homeomorphic to $\RP$. The main result of this section is \cref{lem:makeRP}, which allows us to build such a subgraph from smaller substructures of $\HH$. We first discuss our techniques in \cref{ss:RP-strat}, then state and prove \cref{lem:makeRP} in \cref{ss:makeRP}.

\subsection{Deconstructing $\RP$}\label{ss:RP-strat}

To motivate our strategy for building a copy of $\RP$, let us first describe how S\'os, Erd\H os, and Brown \cite{SoErBr73} located copies of the sphere $\sph$. Suppose $\HH$ is a 3-uniform hypergraph with $n$ vertices and $cn^{5/2}$ edges, where $c$ is a sufficiently large constant. By an averaging argument, one may find vertices $u,u'$ such that $\HH_{u,u'}$ has at least $n$ edges. This implies that there is a cycle $C=v_1\cdots v_k$ in $\HH_{u,u'}$. One obtains a subgraph of $\HH$ homeomorphic to $\sph$ comprising the $2k$ edges
\[
\bigcup_{e\in E(C)}\{ue,u'e\}=\{uv_1v_2,u'v_1v_2,\ldots,uv_{k}v_1,u'v_{k}v_1\}.
\]
Such a triangulation of $\sph$ is called a \emph{double pyramid}.

The double pyramid may be viewed as the union of two interior-disjoint disks with edge sets $\{ue:e\in E(C)\}$ and $\{u'e:e\in E(C)\}$, glued together along their shared boundary $C$. The sphere has a corresponding decomposition as a CW complex, given by attaching two copies of the disk $\disk$ to a circle $\sone$. These decompositions are pictured side by side in \cref{fig:S2CW}.

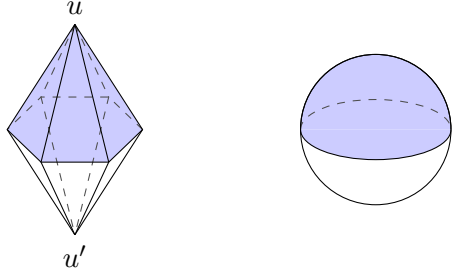
\begin{figure}[t]
\centering
\begin{tikzpicture}[faint/.style={black!70, dashed},shaded/.style={fill=blue!20}]
\begin{scope}[xscale=1.8, scale=0.5]
	\coordinate (u1) at (0,2.8);
	\coordinate (u2) at ($(0,0)-(u1)$);
	\node[above] at (u1) {$u$};
	\node[below] at (u2) {$u'$};

	\foreach \i in {0,...,5} {
		\coordinate (v\i) at (\i*60:1);
	}
	
	\fill[shaded] (u1) -- (v0) -- (v5) -- (v4) -- (v3);
	

	\foreach \i in {1,2} \draw[faint] (u1) -- (v\i) -- (u2);
	\draw[faint] (v0) -- (v1) -- (v2) -- (v3);
	\foreach \i in {0,3,4,5} \draw (u1) -- (v\i) -- (u2);
	\draw (v3) -- (v4) -- (v5) -- (v0);
\end{scope}

\begin{scope}[scale=1,xshift=3cm]
	\coordinate (O) at (0,0);
	\draw[shaded] (O) arc (180:0:1);
	\draw[yscale=0.4,shaded] (O) arc (180:360:1);
	\draw[yscale=0.4, faint] (O) arc (180:0:1);
	\draw (O) arc (180:360+180:1);
\end{scope}
\end{tikzpicture}
\caption{At left, a double pyramid obtained from a cycle of length 6 in $\HH_{u,u'}$. At right, the decomposition of the sphere $\sph$ into two disks attached to $\sone$. The corresponding top disks are shaded in both images.}
\label{fig:S2CW}
\end{figure}

We decompose $\RP$ in a similar fashion, as two copies of $\disk$ attached to $\svs$.
Consider the standard representation of $\RP$ as a disk with boundary glued to itself antipodally --- this is pictured in \cref{fig:RP}. Let $a$ and $b$ be the loops in $\RP$ depicted, with $a$ traversing half the boundary of the disk and $b$ a diameter of the disk. The union of $a$ and $b$, shown in \cref{fig:svs-ab}, is a subspace of $\RP$ homeomorphic to $\svs$. Moreover, $\RP$  can be recovered from this subspace by attaching two copies of $\disk$ --- one corresponding to each semicircular region of \cref{fig:RP} --- to the concatenated loops $ab$ and $a^{-1}b$. This is summarized in the following proposition.

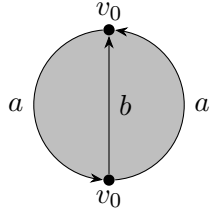
\begin{figure}[t]
\centering
\begin{tikzpicture}[>={Stealth[scale=1]}]
	\fill[black!25] (0,0) circle (1);
	\coordinate[vtx] (v) at (0,1);
	\coordinate[vtx] (w) at (0,-1);
	\node[left] at (-1,0) {$a$};
	\node[right] at (1,0) {$a$};
	\foreach \ang in {4.7} {
		\draw (-1,0) arc (180:90+\ang:1);
		\draw (-1,0) arc (180:270-2*\ang:1);
		\draw (1,0) arc (0:90-2*\ang:1);
		\draw (1,0) arc (0:-90+\ang:1);
		\draw[<-] (90-\ang:1) --++ (-1.78*\ang:1pt);
		\draw[<-] (270-\ang:1) --++ (180-1.78*\ang:1pt);
	}
	\draw[->] (w) -- (v) node[right,pos=0.5]{$b$};
	\node[above] at (v) {$v_0$};
	\node[below] at (w) {$v_0$};
	
%
%

\end{tikzpicture}
\caption{Two loops $a$ and $b$ in $\RP$ based at the point $v_0$. Here, $\RP$ is depicted as a disk with boundary points identified antipodally.
}
\label{fig:RP}
\end{figure}

\begin{prop}\label{prop:makeRP}
Let $a$ and $b$ be the two loops in $\svs$ shown in \cref{fig:svs-ab}. Form a CW complex from $\svs$ by attaching one disk to the loop $ab$ and another disk to the loop $a^{-1}b$. The resulting topological space is homeomorphic to $\RP$.
\end{prop}

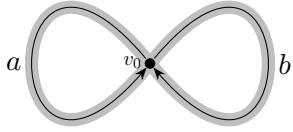
\begin{figure}[t]
\centering
\begin{tikzpicture}[ 
		faint/.style={black!25,line width=5pt},
		bendy/.style={looseness=60},
		>={Stealth[scale=1]}
		]
	\coordinate (SW) at (-2,-1.2);
	\useasboundingbox (SW) rectangle ($-1*(SW)$);
	\foreach \ang/\x/\y in {50/2.5pt/0.5pt} {
		\draw[faint] (\ang:\x) -- (180+\ang:\x);
		\draw[faint] (-\ang:\x) -- (180-\ang:\x);
		\draw[faint] (\ang:\x) to [out=\ang,in=-\ang,bendy] (-\ang:\x);	
		\draw[faint] (180-\ang:\x) to [out=180-\ang,in=\ang-180,bendy] (\ang-180:\x);
		\draw (\ang:\x) to [out=\ang,in=-\ang,bendy] node[right, pos=0.5] {$b$} (-\ang:\x);	
		\draw (180-\ang:\x) to [out=180-\ang,in=\ang-180,bendy] node[left, pos=0.5] {$a$} (\ang-180:\x);
		\draw[<-] (-\ang:\x-\y) -- (-\ang:\x);
		\draw[<-] (\ang-180:\x-\y) -- (\ang-180:\x);
		\fill (0,0) circle (\x-\y);
		\node[scale=0.7] at (-\x-4pt,0.4pt) {$v_0$};
	}
\end{tikzpicture}
\caption{Two loops $a$ and $b$ in $\svs$ sharing the same basepoint $v_0$.}
\label{fig:svs-ab}
\end{figure}

Let $\diskm$ be the quotient of $\disk$ obtained by gluing together two points $x,y$ on the boundary of $\disk$. \Cref{prop:makeRP} decomposes $\RP$ as a union of two copies of $\diskm$ intersecting on their shared boundary, a subspace of $\RP$ homeomorphic to $\svs$.

One might attempt to build $\RP$ using the following na\"ive approach. Choose vertices $u,u'$ and cycles $C,C'\subseteq\HH_{u,u'}$ so that $C$ and $C'$ intersect in a single vertex $v_0$, implying that $C\cup C'$ is homeomorphic to $\svs$. Let $\A,\A'\subseteq\HH$ be subgraphs induced by the edge sets $E(\A)=\{ue:e\in E(C)\cup E(C')\}$ and $E(\A')=\{u'e:e\in E(C)\cup E(C')\}$. One hopes that $\A$ and $\A'$ are copies of $\diskm$ whose union is homeomorphic to $\RP$, and indeed this is almost true. However, the 1-simplex $uv_0$ (resp.\ $u'v_0$) is contained in four different edges of $\A$ (resp.\ $\A'$), so neither $\A$ nor $\A'$ is homeomorphic to $\diskm$.

To obtain a homeomorphic copy of $\RP$, we alter the hypergraphs $\A$ and $\A'$ to avoid these four-way intersections. The resulting construction is pictured in \cref{fig:makeRP}. Let $v_1,v_2$ be the two neighbors of $v_0$ in $C$, and let $v_3,v_4$ be the two neighbors of $v_0$ in $C'$. Consider the edge subsets $D=\{uv_0v_1,uv_0v_3\}\subseteq E(\A)$ and $D'=\{u'v_0v_2,u'v_0v_3\}\subseteq E(\A')$, which correspond to disks with boundaries $v_0v_1uv_3v_0$ and $v_0v_2u'v_3v_0$ respectively. We locate alternate disk subgraphs $\DD,\DD'\subseteq\HH$ having the same boundaries, and replace $D$ and $D'$ with them. If $\DD$ and $\DD'$ are chosen appropriately, the altered hypergraphs $(\A\setminus D)\cup\DD$ and $(\A'\setminus D')\cup\DD'$ are homeomorphic to $\diskm$ with shared boundary $C\cup C'$. In fact, they are created by attaching disks to $C\cup C'$ along the loops $v_0v_2\mydots{C}v_1v_0v_3\mydots{C'}v_4v_0$ and $v_0v_1\mydots{C}v_2v_0v_3\mydots{C'}v_4v_0$, respectively. Using \cref{prop:makeRP}, one can show that their union is homeomorphic to $\RP$.

\Cref{lem:makeRP}, stated in the next subsection, formalizes this construction. It gives explicit conditions on $C,C',\DD,\DD'$ under which the union $(\A\setminus D)\cup\DD\cup(\A'\setminus D')\cup\DD'$ yields a copy of $\RP$ in $\HH$.

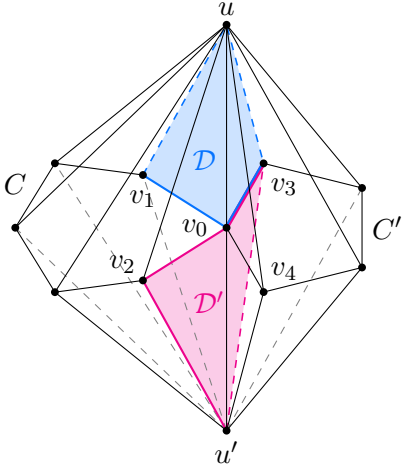
\begin{figure}[t]
\centering
\begin{tikzpicture}[xscale=1.3,scale=0.9]
	\coordinate (v0) at (0,0);
	\coordinate (u1) at (0,3);
	\coordinate (u2) at ($(0,0)-(u1)$);
	\foreach\r [count=\i] in {0.9,1.1,1,1.1,0.9} {
		\coordinate (x\i) at ($(\i*60+0:\r)-(1.4,0)$);
	}
	\foreach\i in {1,...,4} {
		\coordinate (y\i) at ($(\i*72+180:1)+(0.73,0)$);
	}
	\coordinate (v1) at (x1);
	\coordinate (v2) at (x5);
	\coordinate (v3) at (y4);
	\coordinate (v4) at (y1);
	
\foreach \cola/\colb in {blue!55!ProcessBlue/Magenta} {
	
	\fill[\cola,opacity=0.2] (v0) -- (v1) -- (u1) -- (v3) -- (v0);
	\fill[\colb,opacity=0.2] (v0) -- (v2) -- (u2) -- (v3) -- (v0);

	\draw[\cola, semithick, dashed] (v1) -- (u1) -- (v3);
	\draw[\colb, semithick, dashed] (u2) -- (v3);
	
	\foreach \v in {} 
		\draw[gray,dashed] (u1) -- (\v);
	\foreach \v in {y3,x1,x2,x3} 
		\draw[gray,dashed] (u2) -- (\v);

	\draw[\colb,ultra thick] (v0) -- (v3);
    \begin{scope}[overlay]
    	\coordinate (a) at ($2*(v0) - (v3)$);
    	\coordinate (b) at ($2*(v3) - (v0)$);
    	\clip (a) -- (b) -- (a|-b);
      \draw[\cola,ultra thick] (v0) -- (v3);
    \end{scope}
	\draw[\cola, thick] (v0) -- (v1);
	\draw[\colb,thick] (v0) -- (v2) -- (u2);
	
		
	\foreach \v in {x2,x3,x4,x5,v0,y1,y2,y3}
		\draw (u1) -- (\v);
	\foreach \v in {x4,v0,y1,y2} 
		\draw (u2) -- (\v);
	
	\draw (x1) -- (x2) -- (x3) -- (x4) -- (x5);
	\draw (v0) -- (y1) -- (y2) -- (y3) -- (y4);
		
	\foreach \v in {v0,u1,u2,x1,x2,x3,x4,x5,y1,y2,y3,y4}
		\node[scale=0.7,vtx] at (\v) {};
	
	\node[left=3pt] at (v0) {$v_0$};
	\node[below=2pt] at (v1) {$v_1$};
	\node[above left] at ($(v2)+(1pt,-1pt)$) {$v_2$};
	\node[below right] at ($(v3)-(1pt,2pt)$) {$v_3$};
	\node[above right] at ($(v4)+(-1pt,1pt)$) {$v_4$};
	\node[above] at (u1) {$u$};
	\node[below] at (u2) {$u'$};
	
	\node[\cola] at (-0.25,1) {$\DD$};
	\node[\colb] at (-0.2,-1.1) {$\DD'$};
	\node[above=9pt] at (x3) {$C$};
	\node[right] at (v0-|y2) {$C'$};
	
}
\end{tikzpicture}
\caption{Building $\RP$ from cycles $C,C'\subseteq\HH_{u,u'}$ and disks $\DD,\DD'\subseteq\HH$.}
\label{fig:makeRP}
\end{figure}

\subsection{Building a Homeomorphic Copy of $\RP$}\label{ss:makeRP}

We are almost ready to state \cref{lem:makeRP}, which is the focus of this section. However, we must first introduce some basic notation.

Suppose $X$ is a simplicial complex homeomorphic to $\disk$ or $\diskm$. (Recall that $\diskm$ is the quotient of $\disk$ obtained by gluing together two points on the boundary of $\disk$.) Its \emph{boundary} $\partial X$ is a sub-simplicial complex homeomorphic to $\sone$ or $\svs$, respectively. We say that $X$ has \emph{induced boundary} if every simplex $\Delta$ in $X$ with $\Delta\subseteq V(\partial X)$ is also a simplex of $\partial X$. We write $\Vo(X)=V(X)\setminus V(\partial X)$ for the set of \emph{interior vertices} of $X$.
For convenience, if $X=\X(\DD)$ is the simplicial complex associated to a 3-uniform hypergraph $\DD$, we may simply refer to $\partial\DD$ or $\Vo(\DD)$, or say that $\DD$ has induced boundary.

These definitions allow us to state and prove \cref{lem:makeRP}, providing conditions under which a 3-uniform hypergraph contains a subgraph homeomorphic to $\RP$. \cref{fig:makeRP} provides a visual depiction of this lemma.

\begin{lemma}\label{lem:makeRP}
	Let $\HH$ be a 3-uniform hypergraph. Let $u,u'\in V(\HH)$ and write $G=\HH_{u,u'}$. Let $v_0,v_1,v_2,v_3\in V(G)$ be distinct vertices with $v_1,v_2,v_3\in N_G(v_0)$. Suppose
	\begin{enumerate}[label=(\roman*)]
		\item $C$ and $C'$ are cycles in $G$ with $v_1v_0v_2$ a subpath of $C$ and $v_0v_3$ an edge of $C'$; and
		\item $\DD$ and $\DD'$ are subgraphs of $\HH$ homeomorphic to $\disk$ with induced boundaries $\partial\DD=v_0v_1uv_3v_0$ and $\partial\DD'=v_0v_2u'v_3v_0$.
	\end{enumerate}
	Suppose further that the five sets $V(C)\setminus\{v_0,v_1\}$, $V(C')\setminus\{v_0,v_3\}$, $\Vo(\DD)$, $\Vo(\DD')$, and $W=\{u,u',v_0,v_1,v_3\}$ are all disjoint. Then $\HH$ contains a subgraph homeomorphic to $\RP$.
\end{lemma}

\begin{proof}
Write $C=v_0w_1\cdots w_kv_0$ with $w_1=v_1$ and $w_k=v_2$. Write $C'=v_0x_1\cdots x_\ell v_0$ with $x_\ell=v_3$. By hypothesis, the vertices $u,u',v_0,w_1,\ldots,w_k,x_1,\ldots,x_\ell$ are distinct and are not contained in $\Vo(\DD)$ or $\Vo(\DD')$.

Our proof proceeds as follows. We first locate subgraphs $\A,\A'\subseteq\HH$ homeomorphic to $\diskm$ with induced boundaries $\partial \A=\partial \A'=C\cup C'$ and interior vertex sets $\Vo(\A)=\Vo(\DD)\cup\{u\}$ and $\Vo(\A')=\Vo(\DD')\cup\{u'\}$. Then, using \cref{prop:makeRP}, we deduce that the union $\A\cup\A'$ is homeomorphic to $\RP$.

Let $P$ be the path $w_1\cdots w_kv_0x_1\cdots x_\ell=(C\cup C')\setminus \{v_0v_1,v_0v_3\}$ with endpoints $v_1=w_1$ and $v_3=x_\ell$. Consider the subgraph $\A_0\subseteq\HH$ with edge set $E(\A_0)=\{ue:e\in E(P)\}$. Because $P$ repeats no vertices, $\A_0$ is a disk with boundary $\partial \A_0=uv_1\mydots{P}v_3u$. Take $\A=\A_0\cup\DD$. 

We claim that the simplicial complexes $\X(\A_0)$ and $\X(\DD)$ only intersect at the vertex $v_0$ and the path $v_1uv_3$. Because $\DD$ has induced boundary and $V(\A_0)=\{u\}\cup V(P)$ is disjoint from $\Vo(\DD)$, we have $\X(\A_0)\cap \X(\DD)\subseteq\partial \DD=v_0v_1uv_3v_0$. Clearly $\X(\A_0)$ contains $v_0$ and $v_1uv_3$. Moreover, because $v_0v_1$ and $v_0v_3$ are not edges of $P$, they are not 1-simplices of $\X(\A_0)$. Thus, $\X(\A_0)\cap \X(\DD)$ is the disjoint union of $v_0$ and $v_1uv_3$, as described.

We conclude that $\A$ is formed by gluing the disks $\A_0$ and $\DD$ together at a point $v_0$ and disjoint subpath $v_1uv_3$ of their boundaries. Alternatively, $\A$ is formed by attaching the boundary of a disk to the closed walk $v_0v_1\mydots{P}v_3v_0=v_0w_1\cdots w_kv_0x_1\cdots x_\ell v_0$. This yields that $\A$ is homeomorphic to $\diskm$ with $\partial A=C\cup C'$. Moreover, $\Vo(\A)=(V(\A_0)\cup \Vo(\DD)\cup V(\partial\DD))\setminus V(\partial\A)=\{u\}\cup\Vo(\DD)$.

Lastly, we verify that $\A$ has induced boundary, by checking which simplices of $\X(\A_0)$ and $\X(\DD)$ are fully contained in $V(\partial\A)$. Because $u\notin V(\partial\A)$, the only simplices of $\X(\A_0)$ fully contained in $V(\partial\A)$ are the vertices and edges of $P$, which are simplices of $\partial\A=C\cup C'$. Now, notice that $V(\DD)\cap V(\partial\A)=\{v_1,v_0,v_3\}\subseteq V(\partial\DD)$. Because $\DD$ has induced boundary, the only simplices of $\X(\DD)$ fully contained in $V(\partial \A)$ are those simplices of $\partial\DD=v_0v_1uv_3v_0$ fully contained in $\{v_1,v_0,v_3\}$, i.e.\ the vertices and edges of the path $v_1v_0v_3$. These are all simplices of $\partial\A=C\cup C'$, so $\A$ has induced boundary.

The construction of $\A'$ is analogous. Let $P'=w_k\cdots w_1v_0x_1\cdots x_\ell=(C\cup C')\setminus \{v_0v_2,v_0v_3\}$. Let $\A_0'\subseteq\HH$ have edge set $E(\A_0')=\{u'e:e\in E(P')\}$ and take $\A'=\A'_0\cup\DD'$. An analogous argument shows that $\A'$ is formed by attaching the boundary of a disk to the closed walk $v_0v_2\mydots{P'}v_3v_0=v_0w_k\cdots w_1v_0x_1\cdots x_\ell v_0$, and that $\A'$ has induced boundary $\partial\A'=C\cup C'$ and interior vertex set $\Vo(\A')=\{u'\}\cup V(\DD')$.

We end our proof by showing that the union $\A\cup\A'$ is homeomorphic to $\RP$. Because $\A$ and $\A'$ induce their boundaries and have disjoint interior vertex sets $\Vo(\A)=\{u\}\cup\Vo(\DD)$ and $\Vo(\A')=\{u'\}\cup\Vo(\DD')$, the simplicial complexes $\X(\A)$ and $\X(\A')$ only intersect on their shared boundary $C\cup C'$. Thus, $\A\cup\A'$ is formed by attaching two interior-disjoint copies of $\disk$ to $C\cup C'$, identifying the disks' boundaries with the two closed walks $v_0w_1\cdots w_kv_0x_1\cdots x_\ell v_0$ and $v_0w_k\cdots w_1v_0x_1\cdots x_\ell v_0$. By \cref{prop:makeRP}, the resulting topological space is homeomorphic to $\RP$.
\end{proof}

\section{Probabilistic Preliminaries}\label{s:adm}

\Cref{lem:makeRP} reduces \cref{mainthm} to finding substructures $C,C',\DD,\DD'$ of a 3-uniform hypergraph $\HH$ satisfying the relevant hypotheses. We locate these substructures via a probabilistic approach, analyzing the likelihood that a randomly chosen subset of $V(\HH)$ will contain each of these substructures. To quantify these probabilities, we require some new definitions. The terms \emph{admissible} and \emph{semi-admissible} (\cref{dfn:adm-1,dfn:adm-2} below) were introduced by \kptz \cite{KPTZ21} and were crucial to their upper bound on the Tur\'an number of the torus, which also used probabilistic methods.

This section proves three technical lemmas --- \cref{lem:makefineprob,lem:denseadm,lem:findv0} --- which rely on this terminology. When combined, these lemmas allow us to bound the probability that each of the above substructures exists in a randomly chosen subset of $V(\HH)$, paving the way for the proof of \cref{mainthm} in \cref{s:main}.

Let $p\in[0,1]$ and let $V$ be a set. The notation $U\sim_p V$ indicates that $U$ is a random subset of $V$ which independently contains each $v\in V$ with probability $p$. Now, let $G$ be a graph, and let $U\subseteq V(G)$ be a set of vertices. Given $x,y\in V(G)$, a path from $x$ to $y$ \emph{through $U$} is a path from $x$ to $y$ of length at least 2 in $G[U\cup\{x,y\}]$.

\begin{dfn}\label{dfn:adm-1}
Fix $p,\eps\in(0,1]$ and let $k$ be a positive integer. Let $G$ be a graph and $e=xy$ an edge of $G$. Sample $U\sim_p V(G)$ and let $A_e$ be the event that there are at least $k$ internally vertex-disjoint paths from $x$ to $y$ through $U$. We say the edge $e$ is \emph{$(p,\eps,k)$-admissible} if $\Pr[A_e]\geq 1-\eps$.
\end{dfn}

\begin{dfn}\label{dfn:adm-2}
Let $\HH$ be a 3-uniform hypergraph and let $e=xyz$ and $f=x'yz$ be neighboring edges of $\HH$. Say the pair $(e,f)$ is \emph{$(p,\eps,k)$-admissible} if $yz$ is $(p,\eps,k)$-admissible in the graph $\HH_{x,x'}$. Also, for a  positive integer $r$, say that $(e,f)$ is \emph{$(p,\eps,k,r)$-semi-admissible} if there exist at least $r$ edges $g=x''yz$ such that $(e,g)$ and $(g,f)$ are both $(p,\eps,k)$-admissible.
\end{dfn}

We remark that, although \cref{dfn:adm-1,dfn:adm-2} seem convoluted at first glance, they are in fact quite convenient to work with. Our first technical lemma, which locates disks within random subsets of $V(\HH)$, illustrates their use.

Suppose $(e,f)=(xyz,x'yz)$ is a pair of $(p,\eps,k)$-admissible edges in $\HH$. One corollary of the admissibility condition is that that there are many double pyramids $\SS$ containing $e$ and $f$, each obtained from a cycle in $\HH_{x,x'}$ containing $yz$. Cutting any such $\SS$ along the loop $yxzx'y$ partitions $\SS$ into two disks $\{e,f\}$ and $\SS\setminus\{e,f\}$, both having boundary $yxzx'y$. Furthermore, it is not difficult to see that $\DD=\SS\setminus\{e,f\}$ has induced boundary. \Cref{lem:makefineprob} extends this idea to a more general setting: given only a semi-admissibility condition on $(e,f)$, we locate such a disk $\DD$ within a randomly chosen subset of $V(\HH)$.

\begin{lemma}\label{lem:makefineprob}
	Fix $p,\eps\in(0,\frac 12]$ and positive integers $k,r$. Let $\HH$ be a 3-uniform hypergraph. Let $(e,f)=(xyz,x'yz)$ be a pair of $(p,\eps,k+2,r)$-semi-admissible edges in $\HH$, and fix a set $W\subseteq V(\HH)$ with $|W|\leq k$. Randomly sample $U\sim_{2p}V(\HH)$. Then with probability at least $1-2r\eps-(1-2p)^{r-k}$, there is $\DD\subseteq\HH$ homeomorphic to $\disk$ with induced boundary $\partial\DD=yxzx'y$ and interior vertex set $\Vo(\DD)\subseteq U\setminus W$.
\end{lemma}

\begin{proof}
Let $T$ be a set of $r$ vertices such that for each $w\in T$, the triple $g=wyz$ is an edge of $\HH$ with $(e,g)$ and $(g,f)$ both $(p,\eps,k+2)$-admissible.

Sample $U_1\sim_{\frac 12}U$ and let $U_2=U\setminus U_1$. For each $w\in T$, consider the following three events.
\begin{itemize}
	\item[($A_w$)] $w\in U\setminus W$.
	\item[($B_w$)] There is a path $P_w\subseteq\HH_{x,w}$ from $y$ to $z$ through $U_1\setminus (W\cup\{x'\})$.
	\item[($C_w$)] There is a path $Q_w\subseteq\HH_{w,x'}$ from $y$ to $z$ through $U_2\setminus (W\cup\{x\})$.
\end{itemize}
Suppose $A_w$, $B_w$, and $C_w$ hold for some $w\in T$. Write $P_w=ya_1\cdots a_sz$ and $Q_w=yb_1\cdots b_tz$ with $s,t\geq 1$. For convenience, set $a_0=b_0=y$ and $a_{s+1}=b_{t+1}=z$. Consider the subgraph $\DD$ of $H$ with edge set
\[
E(\DD)=\left(\bigcup_{i=0}^s\{xa_ia_{i+1},wa_ia_{i+1}\}\right)
	\cup\left(\bigcup_{j=0}^t\{wb_jb_{j+1},x'b_jb_{j+1}\}\right)
\]
as pictured in \cref{fig:makedisk}. Because the $5+s+t$ vertices pictured are all distinct, it follows that $\DD$ is a disk with boundary $\partial\DD=yxzx'y$. Notice that each simplex of $\X(\DD)$ contains at most one vertex of $\{x,x'\}$ and at most one vertex in $\{y,z\}$, the latter because $s,t\geq 1$. It follows that the only simplices of $\X(\DD)$ that are fully contained in $V(\partial\DD)$ are the 1-simplices $xy,x'y,xz,x'z$ and their vertices, implying that $\DD$ has induced boundary. Lastly, notice that conditions $A_w$, $B_w$, and $C_w$ together imply that
\[\Vo(\DD)=\{w,a_1,\ldots,a_s,b_1,\ldots,b_t\}\subseteq U\setminus W.\]
We conclude that if the events $A_w$, $B_w$, and $C_w$ simultaneously hold for some $w\in T$, then there is a disk $\DD\subseteq\HH$ with induced boundary $\partial\DD=yxzx'y$ and interior vertex set $\Vo(\DD)\subseteq U\setminus W$. It remains to bound the probability that $A_w$, $B_w$, and $C_w$ hold simultaneously for some $w\in T$.

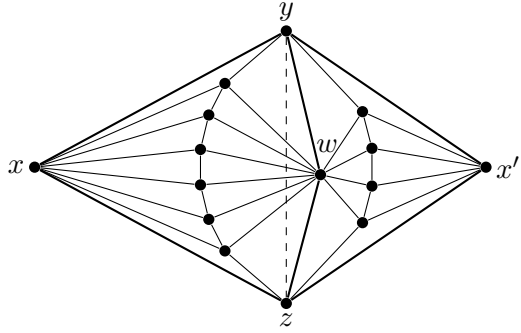
\begin{figure}[t]
\centering
\begin{tikzpicture}
	\foreach \ang in {65} {
		\path (-1,0) arc (0:\ang:2) coordinate[vtx] (y);
		\path (-1,0) arc (0:-\ang:2) coordinate[vtx] (z);
		\coordinate (v0) at (y); \coordinate (w0) at (y);
		\foreach \i [evaluate=\ii using (\i+1.3)/9.6] in {1,...,6} {
			\path (y) arc(-\ang+180:\ang+180:2)
				coordinate[vtx,pos=\ii] (v\i);
		}
		\foreach \i [evaluate=\ii using (\i+2)/9] in {1,...,4} {
			\path (y) arc(\ang:-\ang:2)
				coordinate[vtx,pos=\ii] (w\i);
		}
	}
	\draw[dashed] (y) -- (z);
	\coordinate[vtx] (x1) at (-5.5,0);
	\coordinate[vtx] (x2) at (-1.7,-0.1);
	\coordinate[vtx] (x3) at (0.5,0);
	\foreach \i in {1,...,6}
		\draw (x1) -- (v\i) -- (x2);
	\foreach \i in {1,...,4}
		\draw (x2) -- (w\i) -- (x3);
	\draw[thick] (z) -- (x1) -- (y) -- (x2) -- (z) -- (x3) -- (y);
	\draw (y) -- (v1) -- (v2) -- (v3) -- (v4) -- (v5) -- (v6) -- (z);
	\draw (y) -- (w1) -- (w2) -- (w3) -- (w4) -- (z);
	
	\node[left] at (x1) {$x$};
	\node at ($(x2) + (0.09,0.39)$) {$w$};
	\node[right] at (x3) {$x'$};
	\node[above] at (y) {$y$};
	\node[below] at (z) {$z$};
\end{tikzpicture}
\caption{A disk $\DD\subseteq\HH$ constructed from the semi-admissible pair $(xyz,x'yz)$.}
\label{fig:makedisk}
\end{figure}

Because $(e,g)$ is $(p,\eps,k+2)$-admissible and $U_1\sim_p V$, there are $k+2$ internally vertex-disjoint paths in $\HH_{x,w}$ from $y$ to $z$ through $U_1$ with probability at least $1-\eps$. At least one of these paths avoids every vertex in $W\cup\{x'\}$, so $\Pr[B_w]\geq 1-\eps$. Similarly, $\Pr[C_w]\geq 1-\eps$ for each $w\in T$. A union bound yields that the events $B_w$ and $C_w$ hold simultaneously for all $w\in T$ with probability at least $1-2r\eps$.

It remains to bound the probability that $A_w$ holds for some $w\in T$. This is the probability that $T\cap (U\setminus W)=(T\setminus W)\cap U$ is nonempty, which is $1-(1-2p)^{|T\setminus W|}$. Thus, another union bound yields
\begin{align*}
\Pr[&\exists w\in T\colon A_w,B_w,C_w\text{ all hold}]
\\&\geq\Pr[\text{($B_w$ and $C_w$ hold for all $w\in T$) and ($A_w$ holds for some $w\in T$)}]
\\&\geq 1-2r\eps-(1-2p)^{|T\setminus W|}\geq 1-2r\eps-(1-2p)^{r-k}.\qedhere
\end{align*}

\end{proof}

The remaining two lemmas in this section help us find the vertices $u,u',v_0,v_1,v_2,v_3$ described in \cref{lem:makeRP}. \Cref{lem:denseadm} locates vertices $u,u'$ and a large subgraph $G\subseteq\HH_{u,u'}$ which gives rise to many semi-admissible edge pairs. \Cref{lem:findv0} then identifies incident edges $v_0v_1$ and $v_0v_3$ in $G$ satisfying admissibility conditions in a subgraph $H\subseteq G$. \Cref{lem:findv0} also upper-bounds $\deg_H(v_0)$; this allows us to randomly select $v_2$ from $N_H(v_0)$ while lower-bounding the probability that $v_2$ takes on any fixed value.

\begin{lemma}\label{lem:denseadm}
Fix $p,\eps\in(0,1]$, integers $r,k$, and a large constant $d$. There is a constant $c=c(p,\eps,r,k,d)$ such that the following holds.

Suppose $\HH$ is a 3-uniform hypergraph on $n$ vertices with at least $cn^{5/2}$ edges. Then there are vertices $u,u'\in V(\HH)$ and a graph $G\subseteq\HH_{u,u'}$ with at least $dn/4$ edges such that, for any incident edges $vw,vw'\in E(G)$, the pairs $(uvw,uvw')$ and $(u'vw,u'vw')$ are both $(p,\eps,k,r)$-semi-admissible in $\HH$.
\end{lemma}

The proof of \cref{lem:denseadm} relies on the following lemma of \kptz \cite{KPTZ21}, which finds a positive-density subset of $E(\HH)$ with many semi-admissible pairs.

\begin{lemma}[Lemma 3.3 in \cite{KPTZ21}]
\label{lem:semiadm}
Let $p,\eps\in(0,1]$ and let $k,r,n$ be positive integers. Let $\HH$ be a 3-uniform hypergraph with $n$ vertices and at least $\frac{12r}p\sqrt{\frac k\eps}n^{5/2}$ edges. Then $E(\HH)$ contains a subset $F$ of at least $\frac 12|E(\HH)|$ edges such that any pair of neighboring edges in $F$ is $(p,\eps,k,r)$-semi-admissible in $\HH$.
\end{lemma}

\begin{proof}[Proof of \cref{lem:denseadm}]
Choose $c=\max\left(\frac{12r}p\sqrt{\frac k\eps},\sqrt{\frac d{12}}\right)$. By \cref{lem:semiadm}, there is a subset $F\subseteq E(\HH)$ of size at least $\frac c2n^{5/2}$ with the property that any neighboring edges $xyz,x'yz\in F$ are $(p,\eps,k,r)$-semi-admissible in $\HH$. Let $\HH'\subseteq\HH$ be the 3-uniform hypergraph on $V(\HH)$ with edge set $F$. We use an averaging argument to show that $e(\HH'_{u,u'})\geq dn/4$ for some $u,u'\in V(\HH)$. In this case, $G=\HH'_{u,u'}$ satisfies the desired conditions.

Set $V=V(\HH)$ and randomly choose $(u,u')\in V^2$ uniformly among all pairs of distinct vertices. Given an unordered pair of vertices $vw\in\binom V2$, let $T_{vw}=|\{u''\in V:u''vw\in F\}|$. We have
\begin{align*}
\E[e(\HH_{u,u'})]
&=\sum_{vw\in\binom V2}	\Pr[vw\in E(\HH_{u,u'})]
=\frac 1{n(n-1)}\sum_{vw\in\binom V2}T_{vw}(T_{vw}-1).
\end{align*}
The average of the $T_{vw}$ is
\[
T=\frac 2{n(n-1)}\sum_{vw\in\binom V2}T_{vw}=\frac 2{n(n-1)}\times 3|F|\geq 3c\sqrt n.
\]
Then, by convexity,
\[
\E[e(\HH_{u,u'})]\geq\frac 1{n(n-1)}\times\binom n2T(T-1)\geq\frac 12(3c\sqrt n)(3c\sqrt n-1).
\]
Because $c\geq\max\left(1,\sqrt{\frac d{12}}\right)$ and $n\geq 1$, we have
\[
\E[e(\HH_{u,u'})]\geq\frac 12(3c\sqrt n)(2c\sqrt n)=3c^2n\geq\frac{dn}4.
\]
Thus, there are vertices $u,u'\in V$ such that $G=\HH'_{u,u'}$ has at least $dn/4$ edges. Moreover, for any incident edges $vw,vw'\in E(G)$, the four edges $uvw,uvw',u'vw,u'vw'$ are all in $F$, implying that the pairs $(uvw,uvw')$ and $(u'vw,u'vw')$ are both $(p,\eps,k,r)$-semi-admissible in $\HH$.
\end{proof}

\begin{lemma}\label{lem:findv0}
Fix $p,\eps\in(0,1]$ and fix a positive integer $k$. There is a constant $d=d(p,\eps,k)$ such that the following holds.

Suppose $G$ is a graph with $n$ vertices and at least $dn/4$ edges. Then there is a subgraph $H\subseteq G$ containing two incident edges $v_0v_1,v_0v_3\in E(H)$ which are both $(p,\eps,k)$-admissible in $H$. Moreover, $\deg_H(v_0)\leq d$.
\end{lemma}

Proving \cref{lem:findv0} requires another preliminary result from \cite{KPTZ21}, stated below. It asserts that almost all edges of a sufficiently dense graph are admissible.
\begin{lemma}[Lemma 3.1 in \cite{KPTZ21}]
\label{lem:adm}
Fix $p,\eps\in(0,1]$ and integers $k,n$. If $G$ is a graph on $n$ vertices, all but at most $\frac{2k}{p^2\eps}|V(G)|$ edges of $G$ are $(p,\eps,k)$-admissible.
\end{lemma}

\begin{proof}[Proof of \cref{lem:findv0}]
Set $\al=\frac{2k}{p^2\eps}$ and set $d=4(1+2\al)$.

Let us induct on $n$.
For the base case, suppose $n\leq d$, and take $H=G$. By \cref{lem:adm}, the number of $(p,\eps,k)$-admissible edges in $G$ is at least
\[
e(G)-\al n\geq \left(\frac d4-\al\right)n=(1+\al)n>\frac n2,
\]
so there is some vertex $v_0\in V(G)$ incident to at least two $(p,\eps,k)$-admissible edges $v_0v_1$ and $v_0v_3$. Moreover, $\deg_H v_0<n\leq d$.

Now, suppose that $n>d$. Delete edges so that $\frac{dn}4\leq e(G)\leq \frac{dn}4+1$. Partition $V(G)$ into $V_1=\{v\in V(G):\deg v\leq d\}$ and $V_2=\{v\in V(G):\deg v> d\}$. Observe that 
\[
(d+1)|V_2|\leq\sum_{v\in V_2}\deg v\leq 2e(G)\leq \frac{dn}2 +2<\frac{(d+1)n}2,
\]
where the last inequality holds because $n>d>4$. Thus, $|V_2|<\frac n2$ and $|V_1|>\frac n2$. 

Let $m_1$ be the number of edges of $G$ with at least one endpoint in $V_1$ and let $m_2=e(G)-m_1$ be the number of edges with both endpoints in $V_2$. We have that
\[
m_1+m_2=e(G)\geq \frac d4n=|V_1|+2\al|V_1|+\frac d4|V_2|>|V_1|+\al n+\frac d4|V_2|.
\]
Thus, either $m_1>|V_1|+\al n$ or $m_2>\frac d4|V_2|$.

If $m_1>|V_1|+\al n$, we take $H=G$. By \cref{lem:adm}, at most $\al n$ edges of $G$ are not $(p,\eps,k)$-admissible. Hence, more than $|V_1|$ edges with an endpoint in $V_1$ are $(p,\eps,k)$-admissible. It follows that there is a vertex $v_0\in V_1$ incident to two $(p,\eps,k)$-admissible edges $v_0v_1,v_0v_3\in E(G)$. By definition of $V_1$, we have $\deg v_0\leq d$.

Now, suppose $m_2>\frac d4|V_2|$. In this case $0<|V_2|<\frac n2$, the former because the induced subgraph $G'=G[V_2]$ is nonempty. Thus, we may apply the inductive hypothesis in $G'$. As desired, this yields a subgraph $H\subseteq G'$ containing two $(p,\eps,k)$-admissible incident edges $v_0v_1$ and $v_0v_3$ whose shared endpoint $v_0$ has degree at most $d$ in $H$.
\end{proof}

\section{Proof of \Cref{mainthm}}\label{s:main}

We are now ready to present the proof of \cref{mainthm}, which is restated here as follows.

\begin{theorem}
There is a constant $c$ such that the following holds. If $\HH$ is a 3-uniform hypergraph with at least $c n^{5/2}$ edges, then $\HH$ contains a triangulation of $\RP$.
\end{theorem}

\begin{proof}
Set $p=\frac 16$ and $\eps'=\frac 13$, and let $d=d(p,\eps',2)$ be the constant from \cref{lem:findv0}. Let $r$ be a large integer and $\eps$ a small constant, to be determined later. Let $c=c(p,\eps,r,7,d)$ be the constant from \cref{lem:denseadm}.

Our goal is to find vertices $u,u',v_0,v_1,v_2,v_3$; cycles $C,C'$; and disks $\DD,\DD'$ satisfying the hypotheses of \cref{lem:makeRP}. We encourage the reader to revisit \cref{fig:makeRP}, which depicts this setup. 

Choose $u,u'\in V(\HH)$ and $G\subseteq\HH_{u,u'}$ according to \cref{lem:denseadm}. That is, $e(G)\geq dn/4$ and, for any incident edges $vw,vw'\in E(G)$, the pairs $(uvw,uvw')$ and $(u'vw,u'vw')$ are both $(p,\eps,7,r)$-semi-admissible. 
Applying \cref{lem:findv0} to $G$, we locate a subgraph $H\subseteq G$ containing two incident $(p,\eps',2)$-admissible edges $v_0v_1,v_0v_3 \in E(H)$ whose shared endpoint $v_0$ satisfies $\deg_H(v_0)\leq d$.
Set $W=\{u,u',v_0,v_1,v_3\}$.

We introduce two independent sources of randomness. Partition $V(\HH)=U_1\cup U_2\cup U_3\cup U_4$, such that each vertex $v\in V(\HH)$ is independently contained in $U_1$ or $U_2$ with probability $p=\frac 16$ each, and in $U_3$ or $U_4$ with probability $2p=\frac 13$ each. Independently, choose $v_2\in N_H(v_0)\setminus\{v_1,v_3\}$ uniformly at random. (This set is nonempty: we have $|N_H(v_0)\setminus\{v_1\}|\geq 2$ because $v_0v_1$ is $(p,\eps',2)$-admissible in $H$.)

The following two claims lower-bound the probabilities that substructures $C,C'$ (\cref{claimC}) and $\DD,\DD'$ (\cref{claimD}) exist as described in \cref{lem:makeRP}.

\begin{claim}\label{claimC}
With probability at least $1/3d$, there are cycles $C,C'$ in $H$ satisfying (i) and (ii).
\begin{enumerate}[label=(\roman*)]
	\item $C$ contains $v_1v_0v_2$ as a subpath. Moreover, $V(C)\setminus\{v_0,v_1\}\subseteq U_1\setminus W$.
	\item $C'$ contains the edge $v_0v_3$. Moreover, $V(C')\setminus\{v_0,v_3\}\subseteq U_2\setminus W$.
\end{enumerate}
\end{claim}

\begin{proof}
	Consider the following two events.
	\begin{enumerate}
		\item[($A_1$)] There are two internally vertex-disjoint paths in $H$ from $v_0$ to $v_1$ through $U_1$.
		\item[($A_2$)] There are two internally vertex-disjoint paths in $H$ from $v_0$ to $v_3$ through $U_2$.
	\end{enumerate}
	Because the edges $v_0v_1$ and $v_0v_3$ are both $(p,\eps',2)$-admissible, events $A_1$ and $A_2$ each hold with probability at least $1-\eps'$, and a union bound yields $\Pr[A_1\cap A_2]\geq 1-2\eps'=\frac 13$.
	
	Fix $U_1$ and $U_2$ such that $A_1$ and $A_2$ hold. By $A_1$, there is a path $P\subseteq H$ from $v_0$ to $v_1$ through $U_1$ which avoids $v_3$. Consider the third event
	\begin{enumerate}
		\item[($A_3$)] The second vertex of $P$ is $v_2$, i.e.\ $v_0v_2\in E(P)$.
	\end{enumerate}
	Recall that $v_2$ is chosen uniformly from $N_H(v_0)\setminus\{v_1,v_3\}$, and is independent of the events $A_1$ and $A_2$. It follows that
	\[
	\Pr[A_1\cap A_2\cap A_3]=\Pr[A_1\cap A_2]\times\frac 1{\deg_H(v_0)-2}
	\geq \frac 13\times\frac 1{d-2}>\frac 1{3d}.
	\]
	
	To finish the proof, we assume $A_1$, $A_2$, and $A_3$ hold simultaneously, and use this to construct $C$ and $C'$. Set $C=P\cup v_0v_1$. By $A_2$, there is a path $P'$ from $v_0$ to $v_3$ through $U_2$ which avoids $v_1$; we set $C'=P'\cup v_0v_3$.
	
	Let us check condition (i). By $A_3$, it is immediate that $v_1v_0v_2$ is a subpath of $C$. By construction, $V(P)\subseteq U_1\setminus\{v_3\}$, yielding that
	\[
		V(C)\setminus\{v_0,v_1\}=V(P)\setminus\{v_0,v_1\}\subseteq U_1\setminus\{v_0,v_1,v_3\}. 
	\]
	Moreover, because $C\subseteq H\subseteq\HH_{u,u'}$, we see that $u,u'\notin V(C)$. Thus,
	\[
		V(C)\setminus\{v_0,v_1\}\subseteq U_1\setminus\{v_0,v_1,v_3,u,u'\}=U_1\setminus W,
	\]
	completing the proof of (i). The proof of (ii) is analogous.
\end{proof}

\begin{claim}\label{claimD}
The following two events simultaneously occur with probability at least $1-4r\eps-2\times\left(\frac 23\right)^{r-5}$.
\begin{enumerate}[label=(\roman*)]
	\item There is $\DD\subseteq\HH$ homeomorphic to $\disk$ with induced boundary $\partial\DD=v_0v_1uv_3v_0$ and interior vertex set $\Vo(\DD)\subseteq U_3\setminus W$.
	\item There is $\DD'\subseteq\HH$ homeomorphic to $\disk$ with induced boundary $\partial\DD'=v_0v_2u'v_3v_0$ and interior vertex set $\Vo(\DD')\subseteq U_4\setminus W$.
\end{enumerate}
\end{claim}

\begin{proof}
	We show that (i) and (ii) each occur with probability at least $1-2r\eps-\left(\frac 23\right)^{r-5}$. Then, a union bound implies the desired result.

	(i) is a direct application of \cref{lem:makefineprob}. Recall that $U_3\sim_{2p}V(\HH)$. The pair $(uv_0v_1,uv_0v_3)$ is $(p,\eps,7,r)$-admissible in $\HH$ because $v_0v_1,v_0v_3\in E(H)\subseteq E(G)$. Applying \cref{lem:makefineprob} shows that $\DD$ exists with probability at least $1-2r\eps-\left(\frac 23\right)^{r-5}$.
	
	The bound for (ii) is analogous, but requires more care because $v_2$ is chosen randomly. Fix any $v^*_2\in N_H(v_0)\setminus\{v_1,v_3\}$. The pair $(u'v_0v^*_2,u'v_0v_3)$ is $(p,\eps,7,r)$-admissible in $\HH$ because $v_0v^*_2,v_0v_3\in E(H)\subseteq E(G)$. Moreover, because $v_2$ and $U_4$ are chosen independently, the distribution of $U_4$ conditioned on the event $v_2=v_2^*$ is still $U_4\sim_{2p}V(\HH)$. Applying \cref{lem:makefineprob} while conditioning on the event $v_2=v_2^*$, we have 
	\[
	\Pr[\exists\DD'\text{ as in (ii)}\mid v_2=v_2^*]\geq 1-2r\eps-\left(\frac 23\right)^{r-5}.
	\]
	Averaging over all possible $v_2^*\in N_H(v_0)\setminus\{v_1,v_3\}$ yields the desired bound:
	\[
	\Pr[\exists\DD'\text{ as in (ii)}]=\sum_{v_2^*}\Pr[v_2=v_2^*]\times\Pr[\exists\DD'\text{ as in (ii)}\mid v_2=v_2^*]\geq 1-2r\eps-\left(\frac 23\right)^{r-5}.\qedhere
	\]
\end{proof}

Combining \cref{claimC,claimD} via a union bound shows that the following four structures exist simultaneously with probability at least $\frac 1{3d}-4r\eps-2\times \left(\frac 23\right)^{r-5}$.
\begin{itemize}
	\item A cycle $C\subseteq\HH_{u,u'}$ containing the subpath $v_1v_0v_2$ with $V(C)\setminus\{v_0,v_1\}\subseteq U_1\setminus W$.
	\item A cycle $C'\subseteq\HH_{u,u'}$ containing the subpath $v_3v_0$ with $V(C)\setminus\{v_0,v_3\}\subseteq U_2\setminus W$.
	\item A subgraph $\DD\subseteq \HH$ homeomorphic to $\disk$ with induced boundary $\partial\DD=v_0v_1uv_3v_0$ and interior vertex set $\Vo(\DD)\subseteq U_3\setminus W$.
	\item A subgraph $\DD'\subseteq \HH$ homeomorphic to $\disk$ with induced boundary $\partial\DD'=v_0v_2u'v_3v_0$ and interior vertex set $\Vo(\DD')\subseteq U_4\setminus W$.
\end{itemize}
Fix $r$ large enough that $2\times\left(\frac 23\right)^{r-5}<\frac 1{6d}$. Choose $\eps$ small enough (in terms of $r$) that $4r\eps<\frac 1{6d}$. It follows that
\[
\frac 1{3d}-4r\eps-2\times \left(\frac 23\right)^{r-5}>0.
\]
Thus, for some choice of the random partition $U_1,\ldots,U_4$ and random vertex $v_2$, there are structures $C,C',\DD,\DD'$ satisfying the four bullet points.

We obtain a triangulation of $\RP$ in $\HH$ by applying \cref{lem:makeRP}. By construction, the vertices $v_0,v_1,v_2,v_3$ are distinct. Moreover the five sets $V(C)\setminus\{v_0,v_1\}$, $V(C')\setminus\{v_0,v_3\}$, $\Vo(\DD)$, $\Vo(\DD')$, and $W$ are disjoint, as they are respectively contained in the five disjoint sets $U_1\setminus W$, $U_2\setminus W$, $U_3\setminus W$, $U_4\setminus W$, and $W$. Thus, \cref{lem:makeRP} yields that $\HH$ contains a triangulation of $\RP$.
\end{proof}

\section{Concluding Remarks}\label{s:conclusion}

Although our upper bound on $\exh(n,\RP)$ and the upper bound on $\exh(n,\tor)$ given in \cite{KPTZ21} are both asymptotically $O(n^{5/2})$, we remark that neither proof naturally generalizes to all surfaces. Indeed, our work is tied to the decomposition of $\RP$ given in \cref{s:RP}, whereas the strategy of \cite{KPTZ21} is inherently limited to orientable surfaces. However, there are some fundamental similarities: both results follow a probabilistic approach, and the $n^{5/2}$ term can be traced back to \cref{lem:semiadm} (which is Lemma 3.3 in \cite{KPTZ21}) in both cases. This suggests that there may be a proof of \cref{genthm} that unifies the two strategies.

\begin{question}\label{q:unify}
Can \cref{genthm} be proven without extensive case-specific arguments? That is, is there a single unified approach that works for all orientable and non-orientable surfaces $X$?
\end{question}

A positive answer to \cref{q:unify} would have further implications. Notice that the quantity $\exh(n,X)$ can be defined for any \emph{homogenous} $(r-1)$-dimensional simplicial complex $X$, i.e.\ any simplicial complex $X=\X(\FF)$ homeomorphic to some $r$-uniform hypergraph $\FF$. When $r=3$, Keevash, Long, Narayanan, and Scott \cite{KLNS21} recently showed that $\exh(n,X)=O(n^{2.8})$ for any homogenous 2-dimensional simplicial complex $X$, and conjectured that the exponent could be brought down to $\frac 52$.

\begin{conjecture}\label{conj:anyX}
	Let $X$ be any homogenous 2-dimensional simplicial complex. Then we have $\exh(n,X)=O(n^{5/2})$.
\end{conjecture}
\noindent
This conjecture was reiterated by \kptz \cite{KPTZ21} and previously alluded to by Linial \cite{Li08,Li18}.
An affirmative answer to \cref{q:unify} would mark an important next step towards its proof.

Lastly, we ask about $\exh(n,X)$ in higher dimensions.
\begin{question}
Let $r\geq 4$ and let $X$ be a $(r-1)$-dimensional manifold. What can we say about the asymptotics of $\exh(n,X)$? What if $X$ is an arbitrary homogenous $(r-1)$-dimensional simplicial complex?
\end{question}

\paragraph*{Acknowledgements.} The author is grateful to Nikhil Pandit for suggesting the proof of \cref{prop:makeRP} presented in \cref{ss:RP-strat}. The author is supported by a Fannie and John Hertz Foundation Fellowship and an NSF Graduate Research Fellowship (grant number DGE-1656518).

\end{document}